\documentclass[12pt,fleqn]{amsart}

\usepackage{enumerate}
\usepackage{etex}
\reserveinserts{28}
\usepackage[latin1]{inputenc}
\usepackage{amsmath}
\usepackage{amsfonts}
\usepackage{amssymb}
\usepackage{color}                    
\usepackage{hyperref}                 
\usepackage[T1]{fontenc}

\newtheorem{THEO}{Theorem}[section]
\newtheorem{PROP}[THEO]{Proposition}
\newtheorem{CORO}[THEO]{Corollary}
\newtheorem{LEMM}[THEO]{Lemma}

\newtheorem{DEFI}[THEO]{Definition}

\newtheorem{CLAI}[THEO]{Claim}
\newtheorem{PROB}[THEO]{Problem}

\theoremstyle{definition}
\newtheorem{EXAM}[THEO]{Example}

\def\({\left(}
\def\){\right)}

\def\fA{\mathfrak A}
\def\fB{\mathfrak B}
\def\fC{\mathfrak C}

\def\fF{\mathfrak F}
\newcommand{\sm}{\setminus}

\newcommand{\sub}{\subseteq}

\newcommand{\vf}{\varphi}
\newcommand{\eps}{\varepsilon}

\newcommand{\cA}{{\mathcal A}}

\newcommand{\cE}{{\mathcal E}}
\newcommand{\cF}{{\mathcal F}}
\newcommand{\cG}{{\mathcal G}}

\newcommand{\cJ}{{\mathcal J}}

\newcommand{\cS}{\protect{\mathcal S}}

\newcommand{\cI}{\mathcal I}
\newcommand{\btu}{\bigtriangleup}
 \newcommand{\wh}{\widehat}

\newcommand{\la}{\langle}
\newcommand{\ra}{\rangle}
\def\rto{\rightarrow}

\def\CC{\mathcal{C}}

\def\AA{\mathcal{A}}

\def\FF{\mathcal{F}}

\def\PP{\mathcal{P}}

\def\1{\textbf{1}}

\DeclareMathOperator{\SPAN}{\overline{span}}

\newcommand{\Suc}[2]{\ensuremath{\left({#1}\right)_{{#2}=1}^\infty}}

\def\ult{\operatorname{ult}}
\def\Clop{\operatorname{Clop}}
\def\mod{\operatorname{mod }}
\def\dom{\operatorname{dom}}

\def\w{\omega}

\title[WRN algebras and independent sequences]{Weakly Radon-Nikod\'ym  Boolean algebras and independent sequences}
\author[A. Avil\'{e}s]{Antonio  Avil\'{e}s}
\address{Departamento de Matem\'{a}ticas\\
Facultad de Matem\'{a}ticas\\ Universidad de Murcia\\ 30100 Espinardo, Murcia\\
Spain} \email{avileslo@um.es}
\author[G.\ Mart\'{\i}nez-Cervantes]{Gonzalo Mart\'{\i}nez-Cervantes}
\address{Departamento de Matem\'{a}ticas\\
Facultad de Matem\'{a}ticas\\ Universidad de Murcia\\ 30100 Espinardo, Murcia\\
Spain}
 \email{gonzalo.martinez2@um.es}
\author[G. Plebanek]{Grzegorz Plebanek}
\address{Instytut Matematyczny\\ Uniwersytet Wroc\l awski\\ Pl.\ Grunwaldzki 2/4\\
50-384 Wroc\-\l aw\\ Poland}
\email{grzes@math.uni.wroc.pl}

\subjclass[2010]{Primary 03G05,O6E15,28A60; Secondary 46B22, 46B50.}

\keywords{Boolean algebra, Stone space, independent sequence, weakly Radon-Nikod\'ym compact}
\thanks{A.\ Avil\'{e}s and G.\ Mart\'{\i}nez-Cervantes were partially supported by  the research project 19275/PI/14 funded by Fundaci\'{o}n S\'{e}neca - Agencia de Ciencia y Tecnolog\'{i}a de la Regi\'{o}n de Murcia within the framework of PCTIRM 2011-2014 and by Ministerio de Econom\'{i}a y Competitividad and FEDER (project MTM2014-54182-P). \\
G. Plebanek was partially supported by NCN grant 2013/11/B/ST1/03596 (2014-2017)}

\begin{document}

\begin{abstract}
A compact space is said to be weakly Radon-Nikod\'{y}m (WRN) if it can be weak*-embedded into the dual of a Banach space not containing $\ell_1$. We investigate WRN Boolean algebras, i.e. algebras whose Stone space is WRN compact. We show that the class of WRN algebras and the class of minimally generated algebras are incomparable. In particular, we construct a minimally generated nonWRN Boolean algebra whose Stone space is a separable Rosenthal compactum, answering in this way a question of W. Marciszewski.

We also study questions of J. Rodr\'{i}guez and R. Haydon concerning measures and the existence of nontrivial convergent sequences on WRN compacta, obtaining partial results on some natural subclasses.
\end{abstract}
\maketitle

\section{Introduction}

A compact space is said to be Radon-Nikod\'{y}m (RN) if it is homeomorphic to a weak$^\ast$-compact subset of a dual Banach space with the Radon-Nikod\'{y}m property. C.\ Stegall proved that the dual $X^\ast$ of a Banach space $X$ has the Radon-Nikod\'{y}m property if and only if $X$ is Asplund, i.e.\ every separable subspace of $X$ has separable dual.
The class of RN compacta has been widely studied since it was introduced by I. Namioka in \cite{Nam87}. Some remarkable results concerning this class and its relation with other classes of compact spaces such as Corson compacta and Eberlein compacta can be found in \cite{OSV91}, \cite{Arv02}, \cite{Ste91} and \cite{AK13}.
In a similar way, E. Glasner and M. Megrelishvili defined in \cite{GM12} the class of weakly Radon-Nikod\'ym compact spaces (WRN). These are compact spaces which can be weak$^\ast$ embedded into the dual of a Banach space not containing $\ell_1$. Every Eberlein compact space is RN and every RN compact space is WRN. Glasner and Megrelishvili proved that linearly ordered compact spaces are WRN. Thus, the double arrow is an example of a WRN compact space which is not RN (see \cite[Example 5.9]{Nam87}). On the other hand, it follows from a result of Talagrand \cite{Tal81} that $\beta \w$ is not WRN (another proof of S.\ Todorcevic is included in \cite[Appendix]{GM14};
cf.\ Corollary \ref{wrn:4} below). Some other results about this class can be found in \cite{MC}.

\bigskip

In this paper we introduce a notion of a  weakly Radon-Nikod\'ym Boolean algebra (WRN) ---  an  algebra $\fA$ is  WRN if  its
Stone space is weakly Radon-Nikod\'ym compact. In fact our basic definition is given in a purely Boolean language and says that
$\fA$ is WRN if $\fA$ is generated by a family that may be decomposed into countably many pieces, of which none contains an
infinite independent sequence (see Definition \ref{wrn:1}).

WRN algebras are, in a sense, small and share some properties with the class of minimally generated algebras introduced by  Koppelberg \cite{Ko89}.
We show, however, that there is a WRN algebra which is not minimally generated, see Example \ref{mg:3}. Our example of
a minimally generated algebra which is not WRN is more involved, see section \ref{ex}.  The algebra $\fA$ we construct here has the additional property
that its Stone space $K$ is separable Rosenthal compact. In this way we answer a question communicated to us  by W.\  Marciszewski, see Theorem  \ref{ex:4}.

By strengthening the property WRN we define two seemingly natural subclasses: we call the first one uniformly weakly Radon-Nikod\'ym algebras (UWRN),
see Definition \ref{swrn:1}, and the latter is denoted by SWRN. An algebra $\fA$ is SWRN if it is generated by some family containing no infinite independent sequence,
see Definition \ref{swrn:5}. We show that the classes UWRN and SWRN are incomparable, see Proposition \ref{swrn:6} and
Corollary \ref{ncs:6}.

The class of WRN algebras is hereditary. To show that the class of UWRN is also stable under taking subalgebras we prove in section \ref{pwi} a result on
families containing only short independent sequences; our Theorem \ref{pwi:1} is quantitative in nature and is based on the Sauer-Shelah lemma (cf.\   Appendix A).

One of themes under consideration is properties of finitely additive measures on Boolean algebras.
A family $\cE$ in a Boolean algebra contains an infinite independent sequence if and only if contains a sequence separated by some measure, see Lemma \ref{pr:1}.
Using this fact we show in Theorem \ref{pwi:2} that if $\cE$ contains no infinite independent sequence then the same property has the image of $\cE$ under any Boolean polynomial. In section \ref{me} we discuss  two problems on measures on WRN algebras we were not able to resolve.
They are related to a question posed by J.\ Rodr\'{\i}guez, if every Borel measure defined on a WRN compact space is concentrated on a separable subspace.
 Rodr\'{\i}guez's question is connected with his result on measures of convex WRN compacta, which we enclose in Appendix B.

 Finally, we discuss a problem posed by R. Haydon, if every WRN compact space contains a nontrivial converging sequence (Problem \ref{ncs:1}).
 We prove that if $\fA$ is a UWRN algebra then its Stone space is even sequentially compact (Theorem \ref{ncs:4}).
 This implies that an algebra $\fF$ considered  by  Haydon  \cite{H81} is in SWRN but not in UWRN.

We wish to thank Witold Marciszewski and Jos\'e Rodr\'{\i}guez for fruitful discussions; we are grateful  to J.\ Rodr\'{\i}guez   for his consent to include  Theorem \ref{ap:2}  here.

\section{Preliminaries}\label{pr}

We shall consider abstract Boolean algebras $\fA,\fB,\ldots$, keeping the usual set-theoretic notation. In particular,  $a^c$ denotes the complement of $a\in\fA$, but we shall also write $a^1 = a$ and $a^0 = a^c$ when convenient.
Recall that, an indexed family $\{a_i:i\in I\}$ in an algebra $\fA$ is {\em independent} if
\[\bigcap_{i\in I'}a_i\cap \bigcap_{i\in I''} a_i^c\neq 0,\]
 for every pair $I', I''$ of finite disjoint subsets of $I$.

Given an algebra $\fA$ and any $\cG\sub\fA$, we denote by $\la \cG\ra$ the algebra generated by $\cG$,  the smallest subalgebra of $\fA$ containing $\cG$.
If $\la \cG\ra =\fA$ then $\cG$ is called a generating family.

For an algebra $\fA$, by $\ult(\fA)$ we denote its Stone space (of ultrafilters on $\fA$) and $\fA\ni a\to \widehat{a}\in\Clop(\ult(\fA))$ is the Stone isomorphism
between $\fA$ and the algebra of clopen subsets of its Stone space.

By a measure $\mu$ on an algebra $\fA$ we mean a finitely additive nonnegative probability functional $\fA\to [0,1]$.

\begin{DEFI}\label{pr:0}
Let $\mu$ be a measure on an algebra $\fA$. We say that the measure $\mu$
\begin{enumerate}[(a)]
\item is {\em nonatomic} if for every $\eps>0$ there are $n\ge 1$ and  a finite partition $\{a_1,\ldots, a_n\}$ of ${\bf 1}_\fA$
such that $\mu(a_i)<\eps$ for every $i\le n$.
\item  has {\em countable type} if there is a countable subalgebra $\fB$ of $\fA$ such that for every $a\in\fA$ we have
\[\inf\{\mu(a\btu b): b\in\fB\}=0.\]
\item  is {\em strongly countably determined} if there is a countable subalgebra $\fC \sub \fA$ such that for every
$a\in\fA$ we have
\[ \mu(a)=\sup\{\mu(c): c\in\fC, c\sub a\}.\]
\end{enumerate}
\end{DEFI}

This kind of properties of measures on Boolean algebras are discussed in \cite{BN07} and \cite{KP11}.
Clearly a strongly countably determined measure has countable type; recall that the reverse implication does not hold in general.

We shall say that a family $\cE$ (of sets or elements of some Boolean algebra) is $\eps$-separated by a (finitely additive) measure $\mu$ if
$\mu(a\btu b)\ge\eps$ for all distinct $a,b\in\cE$ (so, in particular, $\cE$ should be contained in the domain of $\mu$).
Note that if we are interested in properties of $\cE$ defined in terms of finite Boolean operations then we can in fact assume that $\mu$ is countably additive:
if $\cE$ is contained in a Boolean algebra $\fA$ then we can switch to the family $\wh{\cE}=\{\wh{a}: a\in\cE\}$ of clopen subsets of the Stone
space $\ult(\fA)$ and consider $\wh{\mu}$, where $\wh{\mu}(\wh{a})=\mu(a)$ for $a\in\fA$.

\begin{LEMM}[\cite{FP04}] \label{pr:1}
If $\eps>0$ and $\cE$ is an infinite family $\eps$-separated by some measure $\mu$ then $\cE$ contains an infinite independent sequence.
\end{LEMM}

Note that, in turn, if $(a_n)_n$ is an independent sequence in some Boolean algebra $\fA$  then there is a measure $\mu$ on $\fA$
such that $\mu(a_n)=1/2$ for every $n$ and $a_n$ are (stochastically) $\mu$-independent. In particular,
if $n\neq k$ then $\mu(a_n\btu a_k)=1/2$ so $a_n$ are separated by $\mu$.

Lemma \ref{pr:1} was proved in \cite{FP04} by a direct argument; it can be also easily derived from the classical lemma due to Rosenthal, which we recall here  in the following form.

\begin{LEMM}[\cite{Ro74}]\label{pr:2}
Let $(A_k)_k$ be a sequence of subsets of some set $T$.  Then

\begin{enumerate}[(i)]
\item either there is an infinite $N\sub\w$ such that the sequence of functions $(\chi_{A_k})_{k\in N}$ converges pointwise,
\item or else  there is an infinite $N\sub\w$ such that the sequence  $({A_k})_{k\in N}$ is independent.
\end{enumerate}
\end{LEMM}

\section{Weakly Radon-Nikod\'{y}m algebras}\label{wrn}

The following definition singles out the class of Boolean algebras that plays a central role in our paper.

\begin{DEFI}\label{wrn:1}
A Boolean algebra $\fA$ is {\em weakly Radon-Nikod\'{y}m} (WRN) if
there is a family $\cG\sub \fA$ generating $\fA$ such that $\cG$ can be written as $\cG=\bigcup_n \cG_n$, where, for every $n$,
 $\cG_n$ contains no infinite independent sequence.
\end{DEFI}

The class of WRN algebras is directly related to the class of weakly Radon-Nikod\'{y}m compacta.
 Glasner and Megrelishvili \cite{GM12} define a compact space to be weakly Radon-Nikodým (WRN) if it is homeomorphic to a weak* compact subset of the dual of a Banach space not containing an isomorphic copy of $\ell_1$.
Using the Davis-Figiel-Johnson-Pe{\l}czy\'{n}ski factorization technique one can easily prove that a compact space $K$ is WRN if and only if the Banach space $\CC(K)$ is weakly precompactly generated.
The concept of weakly precompactly generated Banach spaces was introduced by Haydon  \cite{H81}. A Banach space $X$ is said to be weakly precompactly generated if $X=\SPAN(W)$ for some weakly precompact set $W \sub X$, i.e. for some set $W$ such that every sequence in $W$ has a weakly Cauchy subsequence.
We refer the reader to \cite{MC} for further discussion on WRN compacta and related classes.

The following characterization of WRN Boolean algebras is a consequence of \cite[Lemma 2.7 and Theorem 2.8]{MC}.

\begin{PROP}\label{wrn:2}
The following conditions are equivalent for a Boolean algebra $\fA$:
	\begin{enumerate}[(i)]
		\item  $\fA$ is WRN;
		\item There is a decomposition $\fA=\bigcup_{n\in\omega}\cE_n$ such that, for every $n$,   $\cE_n$
		contains no infinite independent sequence;
		\item the Stone space $\ult(\fA)$ of $\fA$ is weakly Radon-Nikod\'{y}m  compact.
	\end{enumerate}
\end{PROP}

The implication $(i)\to (ii)$ in Proposition \ref{wrn:2} does not seem to be obvious; we shall give in section \ref{pwi} its direct proof. Note that condition $(ii)$ provides an equivalent definition of
WRN algebras that is sometimes more convenient; for instance it yields the following.

\begin{CORO}\label{wrn:3}
If $\fA$ is a WRN algebra then every subalgebra $\fB\sub\fA$ is WRN too.
\end{CORO}

Note that  the algebra $\fA$ itself does not contain an infinite independent sequence if and only if the space $\ult(\fA)$ is scattered, i.e.\
$\fA$ is superatomic.
Clearly every countable algebra is WRN. More generally, if $\fA$ is an interval algebra, that is $\fA=\la \cG\ra $, where the family $\cG$ is linearly ordered then $\fA$ is WRN.
It follows  from Corollary \ref{wrn:3} that every subalgebra of an interval algebra is WRN. This is also a consequence
of a result due to Heindorf \cite{He97}, stating that
an algebra  $\fA$ embeds into some interval algebra if and only if $\fA=\la \cG\ra$, where $\cG$ has the property that any two elements of
$\cG$ are either comparable or disjoint.

To give some examples of algebras that are not WRN note another obvious consequence of  Proposition \ref{wrn:2}.

\begin{CORO}\label{wrn:4}
If $\fA$ is a WRN algebra then $\fA$ contains no uncountable independent sequence.
\end{CORO}

Hence $P(\omega)$, $P(\omega)/{\it fin}$, $\Clop(2^{\omega_1})$ are not WRN; moreover,
no infinite complete algebra can be weakly Radon-Nikod\'{y}m.

\section{Minimally generated algebras}\label{mg}

The notion of minimal extensions of  algebras was introduced by S.\ Koppelberg, see  \cite{Ko89}; the basic facts we mention below can be found in \cite{Ko89} or \cite{BN07}.

If $\fB$ is a subalgebra of a Boolean algebra $\fA$ and $x\in \fA\sm\fB$ then $\fB(x)$ denotes the subalgebra of $\fA$ generated by $\fB\cup\{x\}$, that is
\[\fB(x)=\{(b\cap x)\cup (b'\cap x^c): b,b'\in\fB\}.\]
$\fB(x)$ is said to be
a {\em minimal extension} of $\fB$ if for any algebra $\fC$, if $\fB\sub\fC\sub \fB(x)$ then either $\fC=\fB$ or $\fC=\fB(x)$.
We recall the following basic fact on minimal extensions, see Proposition 3.1 in \cite{Ko89}.

\begin{PROP} \label{mg:1}
Let $\fA$ be a Boolean algebra, $\fB \leq \fA$ a subalgebra and $x \in \fA$. Then  $ \fB(x)$ is a minimal extension of $\fB$ if and only if for every $b \in \fB$, $x \cap b$ or $x \cap b^c$ is in $\fB$.
\end{PROP}
\begin{proof}
Suppose that the extension $\fB\leq \fB(x)$ is minimal. Take $b\in \fB$. If $x \cap b$ is not in $\fB$, then $\fB(x\cap b)=\fB(x)$. Therefore, there exist $a_1,a_2 \in \fB$ such that
\[x= \big( a_1 \cap (x\cap b)\big)  \cup \big(a_2 \cap (x \cap b)^c\big) = (a_1 \cap x \cap b ) \cup ( a_2 \cap (x^c \cup b^c)).\]
 Hence $x \cap b^c = a_2 \cap b^c \in \fB$.

Suppose now that for every $b \in \fB$, either $x \cap b$ or $x \cap b^c$ is in $\fB$. Consider any element $y\in \fB(x)\sm\fB$.
Then there are disjoint $a_1,a_2,a_3 \in \fB$ such that
\[y= (a_1 \cap x ) \cup (a_2 \cap x^c) \cup a_3.\]
Since $a_2 \cap x^c = a_2 \cap (a_1 \cup x^c) = a_2 \cap ((a_1^c \cap x)^c)$, it follows that either $a_1 \cap x$ or $a_2 \cap x^c$ is in $\fB$.
By symmetry, we can assume that $a_2\cap x^c\in\fB$. Then $y$ can be written as $y=(a_1\cap x) \cup c$ where $a_1,c \in \fB$ are disjoint.
It follows that $a_1\cap x\notin\fB$ and $a_1\cap x=y\cap c^c\in\fB(y)$. Since $a_1 ^c \cap x$ must be in $\fB$, we get $x\in\fB(y)$ so $\fB(x)=\fB(y)$,
as required.
\end{proof}

\begin{CORO} \label{mg:1.5}
In the setting of Proposition \ref{mg:1}, if
for every finite $\fB_0\leq \fB$ there is a finite subalgebra $\fB_1$ with $\fB_0\leq \fB_1\leq \fB$ such that
$\fB_1\leq \fB_1(x)$ is a minimal extension  then
$\fB(x)$ is a minimal extension of $\fB$.
\end{CORO}

A Boolean algebra $\fA$ is {\em minimally generated over} $\fB\leq \fA$ if $\fA$ can be written, for some ordinal number $\gamma$, as a continuous union
\[\fA=\bigcup_{\xi<\gamma} \fB_\xi,\]
where $\fB_0=\fB$, and $\fB_{\xi+1}$ is a minimal extension of $\fB_\xi$ for every $\xi<\gamma$. Finally, $\fA$ is said to be minimally generated
if it is minimally generated over the trivial algebra $\fB=\{0,1\}$.

Note that if $\fA$ is minimally generated then its Stone space $\ult(\fA)$ can be seen as a limit of an inverse system $(K_\xi, \pi^\eta_\xi)_{\xi<\eta<\kappa}$, where
$K_0=\{0,1\}$, every $K_\xi$ is compact,  and for every $\xi<\kappa$, the bonding map $\pi^{\xi+1}_\xi: K_{\xi+1}\to K_\xi$ has the property that there is a unique $x_\xi\in K_\xi$ such that
$|(\pi^{\xi+1}_\xi)^{-1}(x_\xi)|=2$, and $|(\pi^{\xi+1}_\xi)^{-1}(x)|=1$ for $x\neq x_\xi$.

Koppelberg \cite{Ko89} proved that a minimally generated algebra cannot contain an uncountable independent sequence.
This was generalized by Borodulin-Nadzieja  \cite{BN07} as follows.

\begin{THEO} \label{mg:2}
Let $\fA$ be a minimally generated algebra. Then every measure $\mu$ on $\fA$ has countable  type.
In particular, $\fA$ contains no uncountable independent sequence and, consequently, $\ult(\fA)$ cannot be mapped onto
$2^{\omega_1}$.
\end{THEO}

Let us remark that Theorem \ref{mg:2} and Corollary \ref{wrn:4} name a common feature  of  the classes of minimally generated algebras and WRN algebras.
Another shared property is that both classes contain all superatomic algebras (that is algebras having scattered Stone spaces).
On the other hand,
we show below that there is a WRN Boolean algebra which is not minimally generated and,  in the next section, we give an example of a minimally generated algebra
which is not weakly Radon-Nikod\'{y}m.

\begin{EXAM}\label{mg:3}
	There exists a WRN Boolean algebra  which is not minimally generated.
\end{EXAM}

\begin{proof}
	Let $\fB$ be the algebra of subsets of $[0,1)$ generated by the chain $\{[0,t): 0<t<1\}$. Then $\fB$ is an interval algebra and therefore is minimally generated, see
	Example 2.1 in \cite{Ko89}
	(note that its Stone space is the familiar split interval).
	
	Koppelberg \cite[Example 1]{Ko88} proved, in particular, that the free product $\fA= \fB \otimes \fB$ is not minimally generated.
	Such a  free product is generated by
	\[ \mathcal{G}= \Big \lbrace [0,a) \times [0,1): 0<a< 1 \Big \rbrace \cup  \Big \lbrace [0,1) \times [0,a):  0<a< 1 \Big \rbrace .\]
	Note that no three  elements of  $\mathcal{G}$ are independent so  $\fA$ is a WRN algebra.
\end{proof}

\section{Example of a minimally generated algebra which is not WRN}\label{ex}

We give in this section a construction of a Boolean algebra with the properties announced by the title.
We shall work in the Cantor set $2^\w$; let $\fA_0=\Clop(2^\w)$. For every partial function $\sigma$ on $\w$ into $2$ we write
\[ \left[ \sigma \right]= \lbrace x \in 2^w : x(i) = \sigma(i) \mbox{ for every } i\in \dom\sigma \rbrace.\]

Let  $T=\lbrace 3n : n \in \w \rbrace$ and let $S(T)$ be the space of all permutations of $T$.

Let $x \in 2^\w$, $\varphi \in S(T)$ be given. We shall define a certain set $A(x,\varphi) \sub 2^\w$.
First define partial functions $\sigma_n (x, \varphi)$ on $\w$ as follows.

\begin{enumerate}[(i)]
\item $\sigma_n (x, \varphi) (i)= x(i) \mbox{ if } i\in 3n \setminus T$;
\item $\sigma_n (x, \varphi) (\varphi(i))= x(\varphi(i)) \mbox{ if } i \in T \mbox{ and } i < 3n $;
\item $\sigma_n (x, \varphi) (\varphi(3n))= (x(\varphi(3n)) + 1) \mod 2 $.
\end{enumerate}

Note that every   $\sigma_n(x, \varphi)$ is defined on the set
\[(3n \setminus T )\cup \varphi(\lbrace i \in T: i \leq 3n \rbrace),\]
so the domain of $\sigma_n(x, \varphi)$ is of size $3n+1$.
We now set
\[A(x, \varphi) = \bigcup_n \left[ \sigma_n (x, \varphi) \right ] .\]

We shall say below that a sequence $(A_n)_n$ of subsets of $2^\w$ converges to a point $x\in 2^\w$ if every neighbourhood of $x$ contains $A_n$ for almost all $n$.

\begin{CLAI}\label{ex:1}
For any $x$ and $\varphi$, $\Suc{\left[ \sigma_n (x, \varphi) \right]}{n}$ is a sequence of disjoint clopen subsets of  $2^\w$ converging to $x$.
\end{CLAI}

\begin{proof}
If $n<k$ then $\sigma_n(x,\vf)(\vf(3n))\neq \sigma_k(x,\vf)(\vf(3n))$ so the clopen sets in question are disjoint. If $\tau$ is any partial function with a finite domain $I$
and $x\in\left[\tau\right]$ then take
$n_0$ such that
\[I\sub \left(3n_0\setminus T\right) \cup \vf(\{i\in T: i<3n_0\}).\]
Then $\sigma_n(x,\vf)$ extends $\tau$ so  $\left[ \sigma_n (x, \varphi) \right]\sub\left[\tau\right]$
for every $n\ge n_0$.
\end{proof}

Let us now fix a Borel bijective map $g: 2^\w \rto S(T)$ (recall that between any two uncountable Polish spaces there is always a Borel isomorphism; the fact
that $g$ is Borel will be needed for the proof of Theorem \ref{ex:4}). For every $x\in 2^\w$ take
$A_x = A(x , g(x))$. We define the desired algebra $\fA$ of subsets of $2^\w$ as the one generated by
 $\fA_0$ together with the family $ \lbrace A_x : x \in 2^\w \rbrace$.

\begin{CLAI}\label{ex:2}
The algebra $\fA$ is minimally generated.
\end{CLAI}

\begin{proof}
Note that for any distinct $x,y \in 2^\w$, $\left[ \sigma_n(x,g(x)) \right]$ and $\left[\sigma_n(y,g(y))\right]$ are sequences of clopen sets converging to
$x$ and $y$, respectively. It follows that either $x\notin A_y$ and then $A_x\cap A_y$ is clopen or, $x\in A_y$ and then $A_x\sm A_y$ is clopen.
Therefore, $\fA$ is minimally generated over $\fA_0$ by Proposition \ref{mg:1} and hence $\fA$ is minimally generated
(since $\fA_0$ is minimally generated because it is countable).
\end{proof}

\begin{CLAI}\label{ex:3}
 The algebra $\fA$ is not weakly Radon-Nikod\'{y}m.
\end{CLAI}

\begin{proof}
Take any decomposition $\fA=\bigcup_{n < \w}\mathcal{A}_n$;
We shall prove that there is $n<\w$ such that $\mathcal{A}_n$ contains  an independent sequence.

Define $\Phi_n = \lbrace g(x): A_x \in \mathcal{A}_n \rbrace $ for every $n< \w$.
Since $\bigcup_n \Phi_n = g(2^\w)=S(T)$ and $S(T)$ is a completely metrizable space (notice that $S(T)$ is a $G_\delta$-set in $T^T$, which is completely metrizable), the Baire Category Theorem asserts that there exists $n_0 < \w $ and a partial function $\psi$ from $\w$ to $\w$  such that $\Phi_{n_0} \cap \left[ \psi \right]$ is dense in $S(T) \cap \left[ \psi \right]$. We can assume that the domain of $\psi$ is $\lbrace 0,3,...,3(p-1) \rbrace$ for some $p$;
 fix also $i_0\in\w$ such that the range of $\psi$ is included in $\lbrace 0,3,...,3(i_0-1) \rbrace$.

Note that, by density,
 for any $i \geq i_0$ there is $x_i\in2^\w$ such that  $ A_{x_i}\in \mathcal{ A}_{n_0}$, $ g(x_i)(3p)=3i$ and
$ g(x_i)(3j)=\psi(3j)$ for every $j<p$.
Passing to a subsequence of $i$'s, we can additionally assume that $x_i$ have constant values for all $n<3p$. Then the following are satisfied:

\begin{enumerate}[(a)]\label{fixed}
\item for every $n < p$ and for every $i, j$,
\[\left[ \sigma_n (x_i, g(x_i)) \right] = \left[ \sigma_n (x_j, g(x_j)) \right].\]
\item there is a partial function $\sigma$ from $\w$ into $2$ with the domain of size $3p$ such that for every $i\ge i_0$ we have
\[ \left[ \sigma_p (x_i, g(x_i)) \right] = \left[ \sigma \right] \cap C_i ^{\epsilon_i} ,\]
where we write $C_i^{\epsilon_i}= \lbrace x \in 2^\w: x(3i)= \epsilon_i \rbrace$ for the corresponding one-dimensional cylinder in $2^\w$.
\end{enumerate}

Let $\mu$ be the canonical product measure on $2^\w$. We shall prove that $\lbrace A_{x_i} \rbrace_{i \geq i_0 }$ is $\varepsilon$-separated for some $\varepsilon >0$.

Note that for every $x_i$ and every $n < \w$,
\[\mu(\left[ \sigma_n (x_i, g(x_i)) \right])=\frac{1}{2^{3n+1}}.\]
Using (a)--(b) above, for distinct $i,j\ge i_0$  we get

\[\mu(A_{x_i}\setminus A_{x_j}) \geq \mu \Big( \Big((\left[ \sigma  \right]\cap C_i ^{\epsilon_i}) \setminus (\left[ \sigma \right]\cap C_j ^{\epsilon_j}) \Big)
\setminus \Big(\bigcup_{n >p } \left[ \sigma_n (x_j, g(x_j)) \right] \Big) \Big) \geq\]
\[ \geq \frac{1}{2^{3p+2}} - \sum_{n >p} \frac{1}{2^{3n+1}} =\frac{1}{2^{3p+2}} - \frac{1}{2^{3p+4}}\frac{1}{1-{2^{-3}}} =\frac{5}{7} \frac{1}{2^{3p+2}} .\]

It follows that the sets $A_{x_i}$ for $i\ge i_0$ are $\varepsilon$-separated with $\varepsilon>0$ so by Lemma \ref{pr:1}
 there is an independent subsequence in $\AA_{n_0}$ and we are done.
\end{proof}

The following result summarises our considerations and gives another property of the Boolean algebra we have constructed.
Recall that a compact space $K$ is Rosenthal compact if $K$ can be embedded into $B_1(X)$, the space of Baire-one functions on a Polish space $X$ equipped with the topology of pointwise convergence. We shall use the following result.

\begin{THEO}\cite[Corollary 4.9]{Debs}\label{ex:3.5}
Every separable compact space consisting of Borel functions over a Polish space is Rosenthal compact.
\end{THEO}

\begin{THEO}\label{ex:4}
There is a minimally generated algebra $\fA$ such that its Stone space $K= \ult(\fA)$ is a separable Rosenthal compactum which is not weakly Radon-Nikod\'{y}m.
\end{THEO}	

\begin{proof}
By Claim \ref{ex:3} the algebra $\fA$ is not WRN so $K= \ult(\fA)$ is not weakly Radon-Nikod\'{y}m compact.
It follows easily from Claim \ref{ex:2} that $\fA_0$ is a dense subalgebra of $\fA$. Hence $K$ has a countable $\pi$-base so is, in particular, separable.
We prove below that $K$ is indeed Rosenthal compact.

Given an ultrafilter $u \in K$,  let $z_u$  be the unique point in $2^\w$ such that
\[ \bigcap\{ C\in\Clop(2^\w): C\in u\}=\lbrace z_u\rbrace.\]
\medskip

\noindent {\sc Claim A.} For every $u\in K$ we have
\[\lbrace y \in 2^\w : z_u \in A_y  \rbrace\sub  \lbrace y \in 2^\w : A_y \in u \rbrace \sub  \lbrace y \in 2^\w : z_u \in A_y  \rbrace\cup \lbrace z_u\rbrace.\]

The first inclusion is clear.
To check the latter, note first that $\overline{A_y}=A_y\cup\lbrace y\rbrace$ for every $y\in 2^\w$ since $A_y$ is the union of clopen sets converging to $y$.
Hence  if $A_y\in u$ and $z_u\notin A_y$ then $y=z_u$ (otherwise, $z_y\notin\overline{A_y}$ which contradicts the definition of $z_u$).
\medskip

\noindent {\sc Claim B.} For every $u\in K$, $\lbrace y \in 2^\w : A_y \in u \rbrace$ is a Borel subset of $2^\w$.
\medskip

By Claim A, it is sufficient to check that for any $z\in 2^\w$ the set $\lbrace y \in 2^\w : z \in A_y  \rbrace$ is Borel.
But
\[ \lbrace y \in 2^\w : z \in A_y  \rbrace=\bigcup_n \lbrace y\in 2^\w: z \in \left[ \sigma_n(y, g(y))\right]  \rbrace,\]
and every set $\lbrace y\in 2^\w: z \in \left[ \sigma_n(y, g(y))\right] \rbrace$ is plainly Borel because the function $g$ is Borel.

Consider  now the following mapping $f:K \rto 2^\w\times 2^{2^\w}$
\[  f(u)=\left( z_u, \chi_{ \lbrace y \in 2^\w : A_y \in u \rbrace }\right).\]
Then $f$ is injective since every ultrafilter $u\in K$ is uniquely determined by the family of generators of the algebra $\fA$ that are in $u$.
It is clear that $f$ is continuous. It follows from  Claim B that $K$ is homeomorphic to a pointwise-compact set of Borel functions on a Polish space.
Since $K$ is separable, $K$ is Rosenthal compact by Theorem \ref{ex:3.5}.
\end{proof}

\section{WRN algebras with stronger properties}\label{swrn}

In this section we introduce two subclasses of WRN algebras; they are defined by natural conditions that are slightly stronger
than that of Definition \ref{wrn:1}.

\begin{DEFI}\label{swrn:1}
	A Boolean algebra $\fA$ is in the class $\cI(n)$, where $n\ge 1$, if $\fA$ is generated by a family $\cG\sub\fA$ such that $\cG$ contains no
	$n+1$ independent elements.
\end{DEFI}

\begin{DEFI}\label{swrn:2}
	A Boolean algebra $\fA$ is {\em uniformly weakly Radon-Nikod\'{y}m} (UWRN) if
 $\fA$ is generated by a family $\cG=\bigcup_{n\in\omega}\cG_n$ such that no $\cG_n $ contains an independent sequence of length $n$.
\end{DEFI}

Note that every interval algebra is in $\cI(1)$.  In turn, the following holds.

\begin{THEO}\label{swrn:3}
	Every Boolean algebra from $\mathcal{I}(1)$ is minimally generated.
\end{THEO}
\begin{proof}
Take a Boolean algebra $\fA \in \mathcal{I}(1)$ and a family $\mathcal{G}$ generating $\fA$  containing no  independent pairs of elements.
We shall check the following.
\medskip

\noindent{\sc Claim.} For every finite $\mathcal{J} \sub \mathcal{G}$ and every $x \in \mathcal{G}$, the extension
$\la \mathcal{J}\ra  \leq \la \mathcal{J}\cup \lbrace x \rbrace\ra $ is minimal.
\medskip

It is clear that Claim holds if $\mathcal{J}=\{y\}$,
since $x,y$ are not independent. We argue by induction on $|\cJ|$.

Suppose that every extension $\langle \mathcal{J} \rangle \leq \langle \mathcal{J}\cup \lbrace x \rbrace \rangle$ is minimal whenever $|\mathcal{J}|=n$.
	Take $x \in \mathcal{G}$ and $\mathcal{J} \sub \mathcal{G}$ with $|\mathcal{J}|=n+1$.
	We prove that the extension $\langle \mathcal{J} \rangle \leq \langle \mathcal{J}\cup \lbrace x \rbrace \rangle$ is also minimal.
Choose $y \in \mathcal{J} $ and set $\mathcal{S}=\mathcal{J} \setminus \lbrace y \rbrace$.
	We are going to prove that for every $z \in \langle \mathcal{J} \rangle $, $z \cap x$ or $z^c \cap x$ is in $\langle \mathcal{J} \rangle$.
	Since $z \in \langle \mathcal{J} \rangle = \langle \mathcal{S}\cup \lbrace y \rbrace \rangle$, we know that
\[ z= (a \cap y) \cup (b \cap y^c) \mbox{ for some  } a,b \in \langle \mathcal{S} \rangle.\]
	Since $|\mathcal{S}|=n$, we know that $a \cap x $ or $a^c \cap x$ is in $\langle \mathcal{S} \rangle$, and $b \cap x$ or $b^c \cap x$ is in $\langle \mathcal{S} \rangle$.
	
Suppose that  $z \cap x \notin \langle \mathcal{J} \rangle$. Since
	\[ z \cap x = (a \cap y \cap x) \cup (b \cap  y^c\cap x),\]
  then either $a\cap y \cap x \notin \langle \cJ \rangle$ or $b \cap y^c \cap x \notin \langle \cJ \rangle$.

Consider the case when $a \cap y \cap x \notin \langle \mathcal{J} \rangle$. Then $a \cap x\notin \la\cS\ra$ so   $a^c \cap x\in \langle \mathcal{S} \rangle$.
Moreover, it follows that $x\cap y\neq 0$ and $y\not\sub x$. Since $y$ and $x$ are not independent, this leaves us two possibilities: either
$x\sub y$ or $x\cup y=1$. Hence $y^c\cap x=0$ or $y^c\cap x=y^c$ so $y^c\cap x\in \la\cJ\ra$ in both cases. By easy calculation we get
$$ z^c=(a\cap y)^c\cap (b\cap y^c)^c=(a^c\cup y^c)\cap (b^c\cup y)=(a^c\cap b^c)\cup (a^c\cap y)\cup (b^c\cap y^c)=
(a^c\cap y)\cup (b^c\cap y^c),$$
and it follows that
\[z^c\cap x=\big( (a^c\cap x)\cap y\big)\cup \big( (y^c\cap x)\cap b^c\big)\in \la\cJ\ra.\]

If $ b \cap  y^c\cap x\notin \langle \mathcal{J} \rangle$ then in a similar way we get $b^c\cap x\in \langle \mathcal{S} \rangle$ and $x\cap y\in \la\cJ\ra$, giving
$z^c\cap x\in \la\cJ\ra$. This finishes the proof of Claim.

Now we conclude the proof of the theorem applying Claim and Corollary \ref{mg:1.5}.
\end{proof}

Note that Example \ref{mg:3} in fact gives the following.

\begin{CORO}\label{swrn:4}
	There exists a Boolean algebra in $\mathcal{I}(2)$ which is not minimally generated.
\end{CORO}

We can also strengthen the condition of Definition \ref{wrn:1} in the following way.

\begin{DEFI}\label{swrn:5}
	Let us say that a Boolean algebra is strongly WRN (SWRN) if it is generated by a family containing no infinite independent sequence.
\end{DEFI}

The classes UWRN and SWRN are incomparable, see  Corollary \ref{ncs:4} and the following result.

\begin{PROP}\label{swrn:6}
	There exists a UWRN algebra which is not SWRN.
\end{PROP}
\begin{proof}
	
	Let $\fA$ be the algebra of clopen sets of a countable product of one point compactifications of a discrete set of cardinality $\w_1$.
	Let $\mathcal{F} = \{e^n_\alpha : n<\omega, \alpha<\omega_1\}$ be the canonical generators of $\fA$ which are independent except for the relation $e^n_\alpha \cap e^n_\beta = 0$ whenever $\alpha\neq \beta$.
	Clearly the algebra $\fA$ is UWRN. We prove below  that it is not SWRN.

	Suppose that $\mathcal{G}$ is a system of generators. It is enough to check that the image of $\cG$ under some quotient contains an infinite independent sequence.
	Express each $e^n_\alpha$ as a Boolean polynomial of generators from $\mathcal{G}$ and in turn each such generator as a Boolean polynomial of generators from $\mathcal{F}$. Let $F_k(e^n_\alpha)$ be the set of all $\beta<\omega_1$ such that $e^k_\beta$ appears in such expression of $e^n_\alpha$. Notice that for every $\alpha<\w_1$, $n<\w$ each set $F_k(e_\alpha^n)$ is finite and, moreover, $F_k(e_\alpha^n)=\emptyset$ for all except finitely many $k < \w$.  By passing, for each $n$, to an uncountable subset $A_n\sub \omega_1$ (by this we mean, making a quotient that makes each $e^n_\alpha$, $\alpha\not\in A_n$ vanish), we can suppose that for every $n$ there is $m_n<\omega$ such that $F_k(e^n_\alpha) = \emptyset$ if $k\geq m_n$ and $|F_k(e^n_\alpha)|<m_n$ if $k<m_n$. Moreover, we can also suppose that each family $\{ F_k(e^n_\alpha) : \alpha\in A_n\}$ is a $\Delta$-system. By removing all roots (that form just a countable set), we can suppose that the family  $\{ F_k(e^n_\alpha) : \alpha\in A_n\}$ is always pairwise disjoint. Now is easy to get $\alpha_n \in A_n$ such that
	$F_k(e^n_{\alpha_n}) \cap F_q(e^m_{\alpha_m}) = \emptyset$ for all $k,q,n,m$ with $n\neq m$. If we make vanish all generators of $\mathcal{F}$ except the $e^n_{\alpha_n}$'s, we will find that one of the generators from $\mathcal{G}$ (call it $g_n$) in the expression of $e^n_{\alpha_n}$ is $e^n_{\alpha_n}$ itself. We found an infinite independent sequence.
\end{proof}

We  give in the next section (see Corollary \ref{pwi:3})  a characterization of UWRN algebras.

\section{Playing with independence}\label{pwi}

In this section we study the behaviour of families containing no infinite independent sequences or independent sequences of size $n$ for some $n\in \w$.
As a consequence of the results obtained in this section we get a characterization of UWRN algebras similar to the one given in Proposition \ref{wrn:3}.
The proof below is based on the Sauer-Shelah Lemma (\cite{Sa72}, \cite{Sh72}) which is recalled in the appendix together with a proof.

\begin{THEO}\label{pwi:1}
	Let $\mathcal{E}$ be a family in some Boolean algebra $\fA$ such that $\mathcal{E}$ contains no independent family of size $n$. Fix $r\geq 1$ and set $$ I(n,r):=\min \lbrace s \in \w: \binom{rs}{0}+ \binom{rs}{1}+ \dots + \binom{rs}{n-1} <2^s \rbrace .$$
	Then, for any Boolean polynomial $p(x_1, x_2, \dots, x_r)$ the family $p(\mathcal{E})= \lbrace p(a_1, \dots , a_r): a_1, \dots a_r \in \mathcal{E} \rbrace$ contains no independent sequence of length $I(n,r)$.
\end{THEO}
\begin{proof}
	
	Suppose $p(\cE)$ contains an independent sequence of length $I(n,r)$. Then there exist
	\begin{align*}
	b_1 &=p(a_{1,1},a_{1,2}, \dots , a_{1,r}), \\
	b_2 & =p(a_{2,1},a_{2,2}, \dots , a_{2,r}), \\
	&  \dots \\
	b_{I(n,r)} & =p(a_{I(n,r),1},a_{I(n,r),2}, \dots , a_{I(n,r),r})
	\end{align*}
	
	such that $b_1, b_2, \dots, b_{I(n,r)}$ is an independent family.	
	Without loss of generality, we may suppose that $\cE = \lbrace a_{i,j}: 1\leq i \leq I(n,r), 1\leq j \leq r \rbrace $.
	Let us put, for convenience, $N = rI(n,r)$ and $\cE=\lbrace a_1 , a_2, \dots, a_N \rbrace$. Since $\la\cE\ra$ contains an independent family of size $I(n,r)$, it must contain at least $2^{I(n,r)}$ atoms.
	Moreover, every atom of $\la\cE\ra $ has a unique representation of the form $a_1^{f(1)} \cap a_2^{f(2)} \cap \dots \cap a_N ^{f(N)}$, where $f \in 2^{\lbrace 1,2,\dots, N\rbrace}$ and for each element $a \in \la \cE\ra$, we denote the complement of $a$ as $a^0$ and $a$ as $a^1$.
	Set
	\[ \cF = \lbrace f \in 2^{\lbrace 1,2,\dots, N\rbrace}: a_1^{f(1)} \cap a_2^{f(2)} \cap \dots \cap a_N ^{f(N)} \mbox{ is an atom of }\la \cE\ra \rbrace.\]
	
	We claim that $|\cF|\leq \binom{N}{0} + \binom{N}{1}+ \dots + \binom{N}{n-1}$.
	If not, by the Sauer-Shelah Lemma (see Lemma \ref{ssl}) there exists a set $S \sub \lbrace 1,2,\dots,N \rbrace$ with $|S|=n$ such that
	\[ \lbrace f|_S : f \in \cF \rbrace = 2^S.\]
	But this means that $\lbrace a_i: i \in S \rbrace$ is an independent family, since for each $f\in 2^S$, the element $\bigcap_{i \in S} a_i^{f(i)} $ is nonempty because it contains an atom.
	This is in contradiction with the hypothesis on $\cE$, so
	\[|\cF|\leq \binom{N}{0} + \binom{N}{1}+ \dots + \binom{N}{n-1}.\]
	
	Since the number of atoms of $\la\cE \ra$ is exactly $|\cF|$,
	we conclude that
	
	\[ 2^{I(n,r)} \leq \binom{N}{0} + \binom{N}{1}+ \dots + \binom{N}{n-1}= \binom{rI(n,r)}{0} + \binom{rI(n,r)}{1}+ \dots + \binom{rI(n,r)}{n-1},\]
	in contradiction with the definition of $I(n,r)$.	
\end{proof}

Our Theorem \ref{pwi:1} has the following counterpart.

\begin{THEO}\label{pwi:2}
	Let $\cE$ be a family in some Boolean algebra $\fA$ such that $\cE$ contains no infinite independent sequence.
	Let, for a fixed $r$,  $p(x_1,x_2,\ldots, x_r)$ be any Boolean polynomial.
	Then the family
	\[p(\cE)=\{p(a_1,\ldots, a_r): a_1,\ldots, a_r\in\cE\}\]
	contains no infinite independent sequence.
\end{THEO}

\begin{proof}
Consider first the polynomial $p(x,y)=x\cap y$.	
Suppose that $p(\cE)$ contains $c_n=a_n\cap b_n$ with $a_n,b_n\in\cE$ such that the sequence $(c_n)_n$ is independent.
	By the remark following Lemma \ref{pr:1} there is a finitely additive probability measure $\mu$ on $\fA$ such that
	$\mu(c_n)=1/2$ and $c_n$'s are stochastically independent with respect to $\mu$.
	
	For $k \neq n$ we have
	\[ 1/4=\mu(c_k\sm c_n)\le \mu(c_k\sm a_n)+\mu(c_k\sm b_n),\]
	so either $\mu(c_k\sm a_n)\ge 1/8$ or  $\mu(c_k\sm b_n)\ge 1/8$.
	Say that the pair $\{k,n\}$ gets the colour $a$ if the first inequality holds and the colour $b$ otherwise. By the Ramsey theorem there is an infinite $N\sub\omega$
	such that whenever $k,n\in N$ are different then $\{k,n\}$ has the same colour;  say that this is  $a$.
	
	It follows that for $k,n\in N$, $k<n$, we have
	\[ \mu(a_k\btu a_n)\ge \mu(a_k\sm a_n)\ge \mu(c_k\sm a_n)\ge 1/8,\]
	so the family $\{a_n: n\in N\}$ is $1/8$-separated by $\mu$. Applying Lemma \ref{pr:1} we get a contradiction.

	We can assume that $\cE$ is closed under taking complements. If we consider the polynomial $p'(x,y)=x \cup y$ then
$p'(x,y)=(p(x^c,y^c))^c$ so the result follows for $p(\cE)$ by the argument above.

The general case follows by induction on the complexity of the Boolean polynomial in question.
\end{proof}

Using Theorems \ref{pwi:1}  we conclude the following analogue of Proposition \ref{wrn:2}.

\begin{CORO}\label{pwi:3}
A Boolean algebra	 $\fA$  is UWRN if and only if there is a decomposition $\fA=\bigcup_{n\in\omega}\cE_n$ such that no $\cE_n$
		contains an independent sequence of length $n$.

Consequently, the class of UWRN  algebras is stable under taking subalgebras.
\end{CORO}

\begin{PROB}
Is the class  SWRN hereditary?
\end{PROB}

\section{Measures on UWRN algebras}\label{me}

The following result and Theorem \ref{mg:2} show another common feature of WRN algebras and minimally generated ones.

\begin{PROP} \label{me:1}
If $\mu$ is a  measure on a WRN algebra $\fA$ then $\mu$ has countable type.
\end{PROP}

\begin{proof}
Suppose otherwise; note that then there is $\eps>0$ and an uncountable  family $\cF$ such that
$\mu(a\btu b)\ge\eps$ for any distinct $a,b\in\cF$.

Since $\fA$ is WRN, we have a decomposition $\fA=\bigcup_n \cE_n$ as in  Proposition \ref{wrn:2}(ii). But then
$\cF\cap \cE_n$ is uncountable for some $n$ and we arrive at a contradiction with Lemma \ref{pr:1}.
\end{proof}

In fact Proposition \ref{me:1} can be generalized to a result on Borel measures defined on WRN compacta, see
Proposition \ref{ap:1}.

\begin{PROP} \label{me:2}
If $\mu$ is a nonatomic measure on $\fA$ and $\fA\in\cI(1)$ then $\mu$ is strongly countably determined.
\end{PROP}

\begin{proof}
By the assumption, $\fA=\la \cG\ra $ where $\cG$ contains no independent pair.

Fix $\eps>0$. There is a finite $\cG_0\sub\cG$ such that $\fB=\la \cG_0\ra$ has all atoms of measure $<\eps$. Take any $g\in\cG$ and consider
$b_0,b_1\in\fB$, where  $b_0$ is the maximal element of $\fB$ contained in $g$, while $b_1$ is the minimal element of $\fB$ containing $g$.
\medskip

{\sc Claim.} $b_1\sm b_0$ is an atom of $\fB$.
\medskip

Indeed, for any $h\in \cG_0$, either $h\sub g$ which implies $h\sub b_0$, or $h\cap g= {\bf 0}$ which gives $h\cap b_1= {\bf 0}$, or $g \sub h$ which implies $b_1 \sub h$, or else
$h\cup g= {\bf 1}$ and in this case $h\supseteq b_1\sm b_0$. So $b_1\sm b_0$ is split by no  $h\in\cG_0$ and hence  it is an atom of $\fB$.
\medskip

It follows from Claim that $\mu(b_1\sm b_0)<\eps$, so
\[ \mu(b_0)= \mu(b_1)-\mu(b_1\sm b_0)\ge \mu(g)-\eps,\]
so $b_0$ approximates $g$ from inside; likewise, $b_1^c\sub g^c$ and $\mu(g^c\sm b_1^c)\le \eps$.
Now, taking a countable $\cG'\sub \cG$ such that $\mu$ is nonatomic on $\fC=\la \cG' \ra$, it follows that
for every $g\in\cG$, we have
\[ \mu(g)=\sup\{\mu(c):c\in\fC, c\sub g\} \mbox{ and } \mu(g^c)=\sup\{\mu(c):c\in\fC, c\sub g^c\}.\]
By a standard argument  we conclude that $\mu(a)=\sup\{\mu(c):c\in\fC, c\sub a\}$ for every $a\in\fA$, so
$\mu$ is  strongly countably determined.
\end{proof}

Borodulin-Nadzieja \cite[4.11,4.12]{BN07} proves that a measure on an algebra that is  minimally generated by a sequence of order type $\omega_1$
is strongly countably determined but this is no longer true for arbitrary minimally generated algebras.

\begin{PROB}\label{me:3}
Is it true that for every $n$ and every algebra $\fA\in\cI(n)$, every nonatomic measure on $\fA$ is strongly countably determined?
\end{PROB}

Note that if the answer to the above problem is positive then  every nonatomic measure on an UWRN algebra  is strongly countably determined.
In turn,  this would imply that if $K$ is a zerodimensional compact space with $\Clop(K)$ being a UWRN algebra then every regular Borel measure on $K$ is
concentrated on a separable subspace of $K$.

\begin{PROB}[Rodr\'{\i}guez] \label{me:4}
Let $K$ be a weakly Radon-Nikod\'{y}m compact space. Is it true that every regular Borel measure on $K$ is concentrated on a separable subspace of $K$?
\end{PROB}

That question is motivated by a result due to J.\ Rodr\'{\i}guez, see Appendix B.

\section{Nontrivial convergent sequences}\label{ncs}

As we  noted above, typical spaces without nontrivial converging sequences, such as $\beta\omega$, are not weakly Radon-Nikod\'{y}m compact.
On the other hand, R.\ Haydon proved that there is  a WRN compact space which is not  sequentially compact.
 As far as we are concerned, the following problem is open:

\begin{PROB}(Haydon, \cite{H81})\label{ncs:1}	
Does every WRN compact space contain a nontrivial convergent sequence?
\end{PROB}

In this section we study Problem \ref{ncs:1} for zero-dimensional compact spaces  related to  UWRN and SWRN Boolean algebras.

We recall the construction of D. H. Fremlin used by Haydon  \cite{H81}. This construction provides an example of a SWRN Boolean algebra $\mathfrak{F}$ such that $\ult(\fF)$ is not sequentially compact.

\begin{EXAM}\label{ncs:2}
	Let $\cG$ be a family of subsets of $\w$ maximal with respect to the condition that for every $A,B \in \cG$ there exists $\varepsilon_1, \varepsilon_2 \in \lbrace 0,1 \rbrace$ such that  $A^{\varepsilon_1} \cap B^{\varepsilon_2} $ is finite.
	Let $\fF$ be the subalgebra of subsets of $\w$ generated by $\cG$ (note that $\cG$ contains all finite subsets of $\w$).
 It is clear that $\cG$ does not contain an infinite independent sequence, so $\fF$ is a SWRN algebra.

 Notice that $\ult(\fF)$ contains a natural copy of $\w$ which consists of principal ultrafilters of $\fF$.
 By the maximality of $\cG$ every infinite $A\sub\w$ is split into two infinite parts by some $G\in\cG$. Consequently,
 the sequence of natural numbers $\w\sub \ult(\fF)$ does not contain a converging subsequence, and therefore $\ult(\fF)$ is not sequentially compact.

Note that  $\ult(\fF)\setminus \w$ is a compact space which
is homeomorphic to the Stone space of  the quotient Boolean algebra $\fA=\fF / fin$.
Then $\fA$ is generated by $\cG^\bullet=\{ G^\bullet: G\in\cG\}$. Since no pair from $\cG^\bullet$ is independent, $\fA$ is  in $\cI(1)$.
J. Bourgain proved that every sequence of nonprincipal ultrafilters in $\ult(\fF)$ contains a convergent subsequence, cf.\ \cite{HS80}.
Thus, $\ult(\fA)$ is sequentially compact.
\end{EXAM}

We shall now generalize Bourgain's idea mentioned above.

\begin{LEMM}\label{ncs:4}
	If $\fA$ is a Boolean algebra in $\cI (n)$ for some $n<\w$, then $\ult(\fA)$ is sequentially compact.
\end{LEMM}
\begin{proof}
Let $\cG$ be a set of generators of $\fA$ such that $\cG$ contains no $n+1$ independent elements.
Take a sequence $(x_n)_n$ in $\ult(\fA)$ and suppose that it admits no convergent subsequence.

For every $k\in\omega$ put $A_k=\{g\in\cG: g\in x_k\}$
and consider the sequence $(A_k)_k$ of subsets of $\cG$.

Note that if $N\sub\w$ is infinite then the subsequence $(\chi_{A_k})_{k\in N}$ cannot converge pointwise; indeed, otherwise,
if $\chi_A=\lim_{k\in N} \chi_{A_k}$ then it is easy to check that there is $x\in\ult(\fA)$ such that $x\supseteq A\cup\{g^c:g\in\cG\sm A\}$
and $\lim_{k\in N} x_k=x$.

Hence, by Lemma \ref{pr:2}, $(A_k)_k$ contains an independent subsequence. For the rest of the proof we can assume that
the whole sequence $(A_k)_k$ is independent. We shall check that this is in contradiction with the assumption $\fA$ being in
$\cI(n)$.

Let $I=\{1,2,\ldots, 2^{n+1}\}$. Take any bijection $\vf$ between $I$ and the finite product $\{0,1\}^{n+1}$ and
let the functions $r_k:I\to 2$ be defined as $r_i=\pi_i\circ\vf$ for $i=1,2,\ldots, n+1$, where
$\pi_i$ is the projection. Then the family $(r_i)_{i\le n+1}$ has the property
that for every $\eps\in \{0,1\}^{n+1}$ there is $k\in I$ such that $\eps=(r_i(k))_{i\le n+1}$.

Now for every $i\le n+1$, using the independence of the sets $A_k$ we can pick
\[ g_i\in \bigcap_{k=1}^{2^{n+1}} A_k^{r_i(k)},\]
where, for a set $A$, we mean $A^1=A$ and $A^0$ is its complement.

We claim that $\cG_0=\{g_i: i\le n+1\}$ is an independent subfamily of $\cG$. Indeed, take any $\eps\in\{0,1\}^{n+1}$ and choose $k\in I$ such that
$\eps=(r_i(k))_{i\le n+1}$. If $r_i(k)=\eps(i)=1$ then $g_i\in A_k$ which means that $g_i\in x_k$. If $r_i(k)=\eps(i)=0$
then $g_i\notin A_k$, i.e.\ $g_i\notin x_k$, which is equivalent to $g_i^c\in x_k$. It follows that
$g_i^{\eps(i)}\in x_k$ for every $i\le n+1$ so $\bigcap_{i\le n+1} g_i^{\eps(i)}\neq 0$. This shows that $\cG_0$ is independent and we get
a contradiction, as required.
\end{proof}

\begin{THEO}\label{ncs:5}
	If $\fA$ is an UWRN Boolean algebra, then $\ult(\fA)$ is sequentially compact.
\end{THEO}

\begin{proof}
Since $\fA$ is  UWRN we have $\fA=\la\cG\ra$, and the decomposition $\cG=\bigcup_n \cG_n$ as in Definition  \ref{swrn:1}.
If we let $\fA_n=\la\cG_n\ra$ for every $n$ we have an obvious embedding
\[\ult (\fA)\to \prod_{n=1}^\infty \ult(\fA_n).\]
We conclude the proof applying Lemma \ref{ncs:4} and the fact that the class of sequentially compact spaces is stable under closed subspaces and countable products.
 \end{proof}

\begin{CORO}\label{ncs:6}
	The Boolean algebra $\fF$ from Example \ref{ncs:2} is SWRN but not UWRN.
\end{CORO}

We remark that  another example of a SWRN Boolean algebra which is not UWRN is given by the well-known example of an Eberlein compact space which is not uniformly Eberlein constructed by Benyamini and Starbird \cite{BS76}.

We finally remark that, although there are WRN algebras which are not SWRN, in order to give an answer to Problem \ref{ncs:1} for zero-dimensional compact spaces, it is enough to consider SWRN algebras because we can use a similar argument as in the proof of Theorem \ref{ncs:5}.

\appendix

\section{Sauer-Shelah Lemma}

The following proof of the Sauer-Shelah Lemma (\cite{Sa72}, \cite{Sh72}) is based on the proof contained in Gil Kalai's blog \cite{GK}.

\begin{LEMM}[Sauer-Shelah]\label{ssl}
Let $N,n$ be natural numbers with $1\le n\le N$ and let $T=\lbrace 1,2,\dots,N \rbrace$.
Then for every family $C \sub 2^T$ with
\[|C|> \binom{N}{0}+ \binom{N}{1}+ \dots + \binom{N}{n-1},\]
there exists a set $S \sub T$ with $|S|=n$ such that $ \lbrace f|_S : f \in C \rbrace = 2^S.$
\end{LEMM}

\begin{proof}
	We first prove the following stronger result:
\medskip
	
\noindent	{	\sc Claim.}
 For every family $C \sub 2^T$ there exists a family of sets $\mathcal{F} \sub \PP(T) $
 such that $|\FF|=|C|$ and
 \[ \lbrace f|_S : f \in C \rbrace = 2^S \mbox{ for any }S \in \FF.\]
	
We check the claim by induction on $|C|$. If $|C|=1$ then take $\mathcal{F}=\lbrace \emptyset \rbrace$.
	Suppose $|C|\geq 2$. Without loss of generality, we may suppose that both the families
\[  C_0= \lbrace f \in C: f(1)=0 \rbrace \mbox{ and }  C_1= \lbrace f \in C: f(1)=1 \rbrace,\]
	are nonempty. Put $T'=T\sm\{1\}$.
	By induction, there exists $\FF_0 \sub \PP( T')$ with $|\FF_0|=|C_0|$ such that
	\[ \lbrace f|_S : f \in C_0 \rbrace = 2^S \mbox{ for any }S \in \FF_0 .\]
	Now take $C_1'= \lbrace f|_{ T' }: f \in C_1 \rbrace$.
	Again by induction,  there exists $\FF_1 \sub \PP(T')$ with $|\FF_1|=|C_1'|$ such that
	\[ \lbrace f|_S : f \in C_1' \rbrace = 2^S \mbox{ for any }S \in \FF_1 .\]
	
	Set
\[\FF= \FF_0 \cup \FF_1 \cup \lbrace S\cup\lbrace 1 \rbrace : S \in \FF_0 \cap \FF_1 \rbrace,\]
	and note that
\[ |\FF|= |\FF_0| + |\FF_1|= |C_0|+|C_1'|=|C_0|+|C_1|=|C|.\]
	Therefore it is enough to prove that $ \lbrace f|_S : f \in C \rbrace = 2^S \mbox{ for any }S \in \FF $,
	but this is a consequence of the properties of $\FF_0$ and $\FF_1$. Thus the claim is proved.
\medskip
	
	Now the lemma follows from the fact that $T$ has exactly $\binom{N}{0}+ \binom{N}{1}+ \dots + \binom{N}{n-1}$ subsets of cardinality smaller than $n$, so
by the assumption on  $|C|$  there exists a set $S \sub T$ with $|S|\ge n$ such that $ \lbrace f|_S : f \in C \rbrace = 2^S.$	
\end{proof}


\section{On measures on  WRN compacta}

We enclose here two results on regular Borel measures defined on WRN compacta that are proved in an unpublished note by Jos\'e Rodr\'{\i}guez.

\begin{PROP} \label{ap:1}
If $K$ is WRN compact and if $\mu$ is a  probability regular Borel measure on $K$ then $\mu$ has countable type
(i.e.\ $L_1(\mu)$ is separable).
\end{PROP}

\begin{proof}
 As we mentioned in section \ref{wrn}, if $K$ is WRN compact then $\CC(K)$ is spanned by some weakly precompact set $W\sub \CC(K)$.
 If we consider the natural embedding $\CC(K)\hookrightarrow L_1(\mu)$ then the image of $W$ is norm-separable. Indeed, otherwise
 for some $\eps>0$ we could find functions $f_n\in W$ such that $\int_K |f_n-f_k|\;{\rm d}\mu\ge\eps$ for $n\neq k$. But then $(f_n)_n$
 admits no weakly Cauchy subsequence, a contradiction.

 Since $W$ is norm-separable in $L_1(\mu)$,  a standard argument gives that $\CC(K)=\SPAN(W)$ is also norm-separable in $L_1(\mu)$.
 But $\CC(K)$ is dense in $L_1(\mu)$ so $L_1(\mu)$ is separable itself.
\end{proof}

Let $X$ be a Banach space and let  $K$ be a $weak^*$-compact subset of the dual unit ball  $B_{X^*}$.
Let $\mu$ be a regular  probability measure on $K$;  denote by  $f: K \to X^*$ the identity function.
Then for every $B \in {\rm Borel}(K)$ there is a vector $\nu(B)=\int_B f \; {\rm d}\mu\in X^*$
which is the Gelfand integral of $f$ on $B$, that is

\begin{equation}\label{eqn:1}
\la \nu(B),x\ra=\int_B x\; {\rm d}\mu
\end{equation}

for every $x\in X$, see \cite[page 53]{DU77}. Here every $x\in X$ is seen as a continuous function $K\ni x^*\to x^*(x)$ on $K$.
In other language,  $\nu(B)$ is the barycenter of a measure $1/\mu(B)\cdot \mu_B$ which is the normalized restriction of $\mu$ to $B$.

\begin{THEO}\label{ap:2}
Suppose that $X$ is a Banach space not containing $\ell^1$ and that the set $K\sub B_{X^*}$ is $weak^*$ compact and convex.
Then for every  probability regular Borel measure $\mu$ on $K$ there is a $weak^*$-closed and $weak^*$-separable set $L \sub K$ such that $\mu(L)=1$.
\end{THEO}

\begin{proof}
Consider the set
\[ S:=\left\{\frac{1}{\mu(B)}\int_B f \, d\mu: \, B \in {\rm Borel}(K), \, \mu(B)>0 \right\}. \]
As above, we write $\nu(B)=\int_B f \; {\rm d}\mu$ for simplicity.
\medskip

\noindent {\sc Claim.} The set $S$ is norm-separable.
\medskip

The space  $L_1(\mu)$ is separable by Proposition \ref{ap:1} so there is  a countable family $\cA$ of Borel subsets of $K$ of positive measure
such that $\inf\{\mu(A\btu B):A\in\cA\}=0$ for every Borel set $B\sub K$. Note that
\begin{equation}\label{eqn:2}
\left\| \nu(B)-\nu(A)\right\|=\sup_{x\in B_X} \left| \int_B x\;{\rm d}\mu - \int_A x\;{\rm d}\mu\right| \le \mu(B\btu A).
\end{equation}

Fix $\eps>0$ and   a Borel set $B$ of positive measure; take $A\in\cA$ such that $\mu(B\btu A)<\eps\cdot \mu(B)$ and
$|1/\mu(B)-1/\mu(A)|<\eps$. Then, using \eqref{eqn:2} we get
$$
\left\| \frac{1}{\mu(B)} \nu(B)- \frac{1}{\mu(A)} \nu(A) \right\| \le
\left\| \frac{1}{\mu(B)} \nu(B)- \frac{1}{\mu(B)} \nu(A)\right\|+
\left\| \frac{1}{\mu(B)} \nu(A)- \frac{1}{\mu(A)} \nu(A)\right\|\le $$
\[\le \frac{1}{\mu(B)} \left\|\nu(B)-\nu(A)\right\| + \left\|\nu(A)\right\| \left| \frac{1}{\mu(B)}-\frac{1}{\mu(A)}\right| \le 2\eps.\]
and this verifies the claim.

Since $S$ is norm-separable in $X^*$, the $weak^*$- closed convex hull $L:=\overline{{\rm co}(S)}^{\w^*}$ is $weak^*$-separable.
Let us check that $L$ fulfills the required properties.

First note that  $L \sub K$. To verify  this  it suffices to check that $S \sub K$. Take any $x^*\in X^* \setminus K$. By the Hahn-Banach
theorem, there is $x\in X$ such that $x^*(x) >\alpha:=\sup\{y^*(x): y^*\in K\}$, therefore
\[
	\left\langle\frac{1}{\mu(B)}\int_B f \; {\rm d}\mu,x\right\rangle
	\stackrel{\eqref{eqn:1}}{=}
	\frac{1}{\mu(B)}\int_B   x  \; {\rm d}\mu \leq \alpha <x^*(x),
\]
for every $B \in {\rm Borel}(K)$ with $\mu(B)>0$. Hence $x^*\not\in S$.

It remains to prove that  $\mu(L)=1$; we achieve it by
 checking that for every $x^*\in K\setminus L$ there is
a $weak^*$-open set $U \sub K$ such that $x^*\in U$ and $\mu(U)=0$. Again, the Hahn-Banach theorem
ensures the existence of $x\in X$ such that
\begin{equation*}
	x^*(x)>\beta:=\sup_{y^*\in S}   y^*(x) \stackrel{\eqref{eqn:1}}{=} \sup\left\{\frac{1}{\mu(B)}\int_B  x  \; {\rm d}\mu: \, B \in {\rm Borel}(K), \, \mu(B)>0 \right\}.
\end{equation*}

Fix $\beta <\gamma <x^*(x)$. Then $x^*$ belongs to the $w^*$-open set $U:=\{y^*\in K:y^*(x)>\gamma\}$
and $\int_U x  \; {\rm d}\mu \geq \gamma\mu(U)$. On the other hand, by the very definition of~$\beta$
we also have $\beta\mu(U)\geq \int_U x  \; {\rm d}\mu$. Then $\beta\mu(U)\geq \gamma\mu(U)$ and so $\mu(U)=0$.
\end{proof}

Let us remark that the proof of Theorem \ref{ap:2} actually shows   the following assertion:
{\em Given a convex $weak^*$-compact subset $K$ of a dual Banach space, any regular Borel probability measure of  countable type defined on $K$ is concentrated on a
 $weak^*$-separable subset of K}.

\end{document}